\documentclass[letterpaper,reqno,12pt]{amsart}

\usepackage[margin=1.2in]{geometry}

\usepackage{amssymb, enumerate}
\usepackage{mathrsfs}
\usepackage[all]{xy}
\usepackage{hyperref}
\hypersetup{colorlinks=true}

\theoremstyle{plain}
\newtheorem{thm}{Theorem}[section] 
\newtheorem{cor}[thm]{Corollary}
\newtheorem{prop}[thm]{Proposition}
\newtheorem{conj}[thm]{Conjecture}
\newtheorem{lem}[thm]{Lemma}

\theoremstyle{definition} 
\newtheorem{defn}[thm]{Definition}

\newtheorem*{convention}{Convention}

\theoremstyle{remark}
\newtheorem{rem}[thm]{Remark}

\newtheorem*{acknowledgement}{Acknowledgments}

\newcommand{\F}{\mathbb{F}}
\newcommand{\N}{\mathbb{N}}
\newcommand{\Q}{\mathbb{Q}} 
\newcommand{\C}{\mathbb{C}} 
\newcommand{\R}{\mathbb{R}} 
\newcommand{\Z}{\mathbb{Z}}

\newcommand{\D}{\Delta}

\newcommand{\fa}{\mathfrak{a}}
\newcommand{\fb}{\mathfrak{b}}

\newcommand{\m}{\mathfrak{m}}

\newcommand{\p}{\mathfrak{p}}
\newcommand{\fq}{\mathfrak{q}}
\newcommand{\U}{\mathfrak{U}}

\newcommand{\fpt}{\mathrm{fpt}}
\newcommand{\FPT}{\mathrm{FPT}}
\newcommand{\lct}{\mathrm{lct}}

\newcommand{\ulim}{\operatorname{ulim}}

\newcommand{\SFRR}{\mathrm{SFRR}}
\newcommand{\Acc}{\mathrm{Accum}}
\newcommand{\TT}{\mathcal{T}}

\newcommand{\ultra}[1]{{}^* #1}
\newcommand{\cata}[1]{#1_\#}
\newcommand{\catae}[1]{[ #1 ]_m}
\newcommand{\age}[1]{\lceil #1 \rceil}
\newcommand{\sage}[1]{\lfloor #1 \rfloor}
\newcommand{\fr}[1]{\{ #1 \}}
\newcommand{\qadic}[3][q]{\langle #2 \rangle_{#3, #1}}

\def\ge{\geqslant}
\def\le{\leqslant}
\def\phi{\varphi}
\def\epsilon{\varepsilon}

\def\to{\longrightarrow}

\def\mod{\operatorname{\,mod}}
\def\Hom{\operatorname{Hom}}
\def\Spec{\operatorname{Spec}}

\newsavebox{\circlebox}
\savebox{\circlebox}{\fontencoding{OMS}\selectfont\Large\char13}
\newlength{\circleboxwdht}

\title{On accumulation points of $F$-pure thresholds on regular local rings}

\author{Kenta Sato}
\address{iTHEMS Research Program, RIKEN, 2-1, Wako, Saitama 351-0198, Japan}
\email{kenta.sato.gz@riken.jp}

\thanks{}


\begin{document}

\begin{abstract}
Blickle, Musta\c{t}\u{a} and Smith proposed two conjectures on the limits of $F$-pure thresholds.
One conjecture asks whether or not the limit of a sequence of $F$-pure thresholds of principal ideals on regular local rings of fixed dimension can be written as an $F$-pure thresshold in lower dimension.
Another conjecture predicts that any $F$-pure threshold of a formal power series can be written as the $F$-pure threshold of a polynomial.
In this paper, we prove that the first conjecture has a counterexample but a weaker statement still holds.
We also give a partial affirmative answer to the second conjecture.
\end{abstract}

\maketitle
\markboth{K.~SATO}{ON ACCUMULATION POINTS OF $F$-PURE THRESHOLDS}

\section{Introduction}


In characteristic zero, log canonical thresholds and their limits have played an increasingly important role in birational geometry.
Recently, Hacon, M\textsuperscript{c}Kernan and Xu verified the ascending chain condition for the set of log canonical thresholds in their celebrated paper \cite{HMX14}, which was applied to the termination of flips (\cite{Bir07}) and the boundedness of log Fano varieties (\cite{Bir16}).

On the other hand, there is another remarkable property on the limits of sequences of log canonical thresholds.
For an integer $d>0$, we denote by $\TT_d \subseteq \Q$ the set of all log canonical thresholds $\lct(X;D)$ of a non-zero effective divisor $D$ on a $d$-dimensional smooth variety $X$ over $\C$.
Then the following property was predicted by Koll\'{a}r (\cite{Kol97}), proved by de Fernex, Ein and Musta\c{t}\u{a} (\cite{dFEM10}) and generalized to singular case in \cite{HMX14}.
\begin{thm}[\textup{\cite{dFEM10}}]\label{1.1}
Let $d>1$ be an integer.
Then any accumulation point of the set $\TT_d$ is contained in $\TT_{d-1}$.
\end{thm}

In this paper, we work in positive characteristic and consider an analogous problem for \textit{$F$-pure thresholds} (see Definition \ref{fpt def} below).
For a prime number $p>0$ and an integer $d>0$, we denote by $\TT_{d,p} \subseteq \Q$ the set of all $F$-pure thresholds $\fpt(A;(f))$ of a principal ideal $(f) \subsetneq A$ of an $F$-finite $d$-dimensional regular local ring $A$ of characteristic $p>0$, here $A$ is said to be \textit{$F$-finite} if the Frobenius map $F: A \to A$ is finite.
Motivated by several studies (\cite{TW04}, \cite{Tak13}, \cite{HnBWZ16}) which reveals a strong connection between $F$-pure thresholds in positive characteristic and log canonical thresholds in characteristic zero, the author recently verified the ascending chain condition for the set $\TT_{d,p}$ (\cite{Sat19a}, \cite{Sat19b}).

On the other hand, very little is known for the limits of descending sequences in $\TT_{d,p}$.
In the first half of this paper, we give partial answers to the following conjecture proposed by Blickle, Musta\c{t}\u{a} and Smith as a positive characteristic analogue of Theorem \ref{1.1}.

\begin{conj}[\textup{\cite{BMS09}}]\label{1.2}
Let $p>0$ be a prime number and $d>1$ be an integer.
Then any accumulation point of the set $\TT_{d,p}$ is contained in $\TT_{d-1,p}$.
\end{conj}

By focusing on the $F$-pure thresholds of principal ideals generated by homogeneous polynomials, we first give a sufficient condition for a rational number to be an accumulation point of the set $\TT_{d,p}$ in terms of the degree of the polynomial (Proposition \ref{making accumulation}).
Making use of this condition, we construct a counterexample to the conjecture.

\begin{thm}[Corollary \ref{Counter ex}]\label{1.3}
For any prime number $p>0$, Conjecture \ref{1.2} fails in $d=2$, that is, there exists a rational number $t_p$ such that 
\[
t_p \in \Acc(\TT_{2,p}) \setminus \TT_{1,p},
\]
where $\Acc(\TT_{2,p})$ denotes the set of all accumulation points of $\TT_{2,p}$.
\end{thm}

We next ask what additional restriction makes the conjecture true.
Since the denominator of the rational number $t_p$ in Theorem \ref{1.3} is divisible by $p$, it is natural to ask what if we consider $\Acc(\TT_{2,p}) \cap \Z_{(p)}$ instead of $\Acc(\TT_{2,p})$.
In this case, we obtain the following result.

\begin{thm}[Corollary \ref{Accum positive}]\label{1.4}
Let $p>0$ be a prime number and $d>1$ be an integer.
Then we have
\[
\Acc(\TT_{d,p}) \cap \Z_{(p)} \subseteq \TT_{d-1,p}.
\]
\end{thm}

In the second half of this paper, we study another conjecture which was also proposed in \cite{BMS09}.
Let $\overline{\F}_p$ be an algebraic closure of the finite field $\F_p$ and $\TT_{d,p}^{\circ}$ be the set of all $F$-pure thresholds $\fpt(A; (f))$ of a principal ideal $(f)$ of the regular local ring $A:=(\overline{\F}_p[x_1,\dots,x_d])_{(x_1,\dots,x_d)}$.
Blickle, Musta\c{t}\u{a} and Smith proved that the closure of $\TT_{d,p}^{\circ}$ in $\R$ coincides with the set $\TT_{d,p}$ and furthermore, made the following conjecture.

\begin{conj}\label{1.5}
For a prime number $p>0$ and an integer $d>0$, we have
\[
\TT_{d,p} = \TT_{d,p}^{\circ}.
\]
\end{conj}

Since $F$-pure thresholds of polynomials are sometimes easier to study than those of formal power series (\cite{BMS08}, \cite{BSTZ10}), the above conjecture may be helpful for further understanding of the set $\TT_{d,p}$.
As a corollary of Theorem \ref{1.4}, we give an affirmative answer to the conjecture after restricting to $\Z_{(p)}$.

\begin{cor}[Corollary \ref{fpt polynomial}]\label{1.6}
For a prime number $p>0$ and an integer $d>0$, we have
\[
\TT_{d,p} \cap \Z_{(p)} \subseteq \TT_{d,p}^{\circ}
\]
\end{cor}

In fact, since the closure $\overline{S}$ of a subset $S \subseteq \R$ coincides with $\Acc(S) \cup S$, we have 
\[
\TT_{d,p} = \overline{\TT_{d,p}^{\circ}} = \Acc(\TT_{d,p}^{\circ}) \cup \TT_{d,p}^{\circ} \subseteq \Acc(\TT_{d,p}) \cup \TT_{d,p}^{\circ}.
\]
Combining with Theorem \ref{1.4}, we have 
\[
\TT_{d,p} \cap \Z_{(p)} \subseteq (\TT_{d-1, p} \cap \Z_{(p)}) \cup \TT_{d,p}^{\circ}.
\]
Then Theorem \ref{1.6} follows from induction on $d$.

Finally, we again consider Conjecture \ref{1.5} without restricting to $\Z_{(p)}$.
In this case, the above argument cannot work due to the pathologies encountered in Theorem \ref{1.3}.
However, we give an affirmative answer in the case where $d=2$.

\begin{thm}[Corollary \ref{two-dim}]\label{1.7}
For any prime number $p>0$, Conjecture \ref{1.5} is true in $d=2$, that is, we have
\[
\TT_{2,p} = \TT_{2,p}^{\circ}.
\]
\end{thm}


\begin{small}
\begin{acknowledgement}
The author wishes to express his gratitude to Professor Shunsuke Takagi for his encouragement, valuable advice and suggestions.
This work was supported by RIKEN iTHEMS Program.
\end{acknowledgement}
\end{small}

\begin{convention}
Throughout this paper, all rings are assumed to be commutative and with a unit element. 
\end{convention}

\section{Preliminaries}

\subsection{$F$-pure thresholds}
A ring $R$ of characteristic $p>0$ is said to be \emph{$F$-finite} if the Frobenius morphism $F: R \to R$ is a finite ring homomorphism.
If $R$ is an $F$-finite Noetherian normal local ring, then $R$ is excellent (\cite{Kun76}) and $X=\Spec(R)$ has a \emph{canonical divisor} $K_X$ (see for example \cite[p.4]{ST}).

Through this paper, all rings will be assumed to be $F$-finite and of characteristic $p>0$.

\begin{defn}
A \emph{pair} $(R, \D)$ consists of an $F$-finite Noetherian normal local ring $(R, \m)$ and an effective $\Q$-Weil divisor $\D$ on $\Spec R$.
A \emph{triple} $(R, \D, \fa_\bullet^{t_\bullet} = \prod_{i=1}^m \fa_i^{t_i})$, consists of a pair $(R, \D)$ and a symbol $\fa_\bullet^{t_\bullet}= \prod_{i=1}^m \fa_i^{t_i}$, where $m>0$ is an integer, $\fa_1, \dots, \fa_m \subseteq R$ are ideals, and $t_1, \dots, t_m \ge 0$ are real numbers. 
\end{defn}

\begin{defn}
Let $(R,\D, \fa_\bullet^{t_\bullet} = \prod_{i=1}^m \fa_i^{t_i})$ be a triple.
\begin{enumerate}
\item[$(1)$]
$(R, \D, \fa_\bullet^{t_\bullet})$ is said to be \emph{sharply $F$-pure} if there exist an integer $e>0$ and a morphism $\phi \in \Hom_R(F^e_* R( \age{(p^e-1) \Delta}), R)$ such that 
\[
\phi( F^e_* (\prod_{i=1}^m \fa_i^{\age{t_i (p^e-1)}} ))=R.
\]
\item[$(2)$]
$(R, \D, \fa_\bullet^{t_\bullet})$ is said to be \emph{strongly $F$-regular} if for every non-zero element $c \in R$, there exist an integer $e>0$ and a morphism $\phi \in \Hom_R(F^e_* R( \age{(p^e-1) \Delta}), R)$ such that 
\[
\phi( F^e_* (c \prod_{i=1}^m \fa_i^{\age{t_i (p^e-1)}} ))=R.
\]
\item[$(3)$]
A pair $(R,\D)$ is said to be \emph{sharply $F$-pure} (resp. \emph{strongly $F$-regular}) if so is the triple $(R, \D, R^0)$.
\item[$(4)$] 
An $F$-finite Noetherian normal local ring $R$ is \emph{sharply $F$-pure} (resp. \emph{strongly $F$-regular}) if so is the pair $(R, 0)$.
\end{enumerate}
\end{defn}

\begin{defn}\label{fpt def}
Let $(R, \D, \fa_{\bullet}^{t_\bullet})$ be a sharply $F$-pure triple and $\fb \subsetneq R$ be a non-zero proper ideal. 
\begin{enumerate}
\item[$(1)$] We define the \emph{$F$-pure threshold} of $\fb$ with respect to $(R,\D, \fa_{\bullet}^{t_\bullet})$ as
\[
\fpt(R, \D, \fa_\bullet^{t_\bullet}; \fb) := \sup\left\{ s \ge 0 \mid (R, \D, \fa_{\bullet}^{t_\bullet} \fb^s) \textup{ is sharply $F$-pure} \right\} \in \R_{\ge 0}.
\]
If $\fb$ is the zero ideal, then we define $\fpt(R,\D,\fa_\bullet^{t_\bullet}; \fb) =0$.
\item[$(2)$] For a sharply $F$-pure pair $(R,\D)$, we define $\fpt(R,\D;\fb) : = \fpt(R,\D,R^0; \fb)$.
Moreover, if $\D=0$, we simply denote it by $\fpt(R;\fb)$.
\end{enumerate}
\end{defn}

\begin{lem}\label{F-sing basic}
Let $(R, \D, \fa_\bullet^{t_\bullet})$ be a strongly $F$-regular triple and $\fb \subsetneq R$ be an ideal.
Then we have 
\[
\fpt(R, \D, \fa_\bullet^{t_\bullet}; \fb) := \sup\left\{ s \ge 0 \mid (R, \D, \fa_{\bullet}^{t_\bullet} \fb^s) \textup{ is strongly $F$-regular} \right\}.
\]
\end{lem}

\begin{proof}
The proof is similar to that of \cite[Proposition 2.2 (5)]{TW04}.
\end{proof}

\subsection{Test ideals}

In this subsection, we recall the definition and some basic properties of test ideals.

\begin{defn}
Let $(R, \D,  \fa_\bullet^{t_\bullet} = \prod_{i=1}^m \fa_i^{t_i})$ be a triple.
An ideal $J \subseteq R$ is \emph{uniformly $(\D,  \fa_\bullet^{t_\bullet} , F)$-compatible} if 
$\phi( F^e_* (\fa_1^{\age{t_1 (p^e-1)}}  \cdots \fa_m^{\age{t_m (p^e-1)}} J)) \subseteq J$ 
for every $e \ge 0$ and every $\phi \in \Hom_R(F^e_* R(\age{(p^e-1)\D}), R)$.
\end{defn}

\begin{defn}\label{test def}
Let $(R, \D,  \fa_\bullet^{t_\bullet} =\prod_{i=1}^m \fa_i^{t_i})$ be a triple.
Assume that $\fa_1, \dots, \fa_m$ are non-zero ideals.
Then we define the \emph{test ideal} 
\[
\tau(R, \D,  \fa_\bullet^{t_\bullet} )=\tau(R,\D, \prod_{i=1}^m \fa_i^{t_i}) = \tau(R,\D, \fa_1^{t_1} \cdots \fa_m^{t_m})
\] to be an unique minimal non-zero uniformly $(\D,  \fa_\bullet^{t_\bullet}, F)$-compatible ideal.
The test ideal always exists (see \cite[Theorem 6.3]{Sch10}).

For a pair $(R,\D)$, we define $\tau(R,\D) : = \tau(R,\D,R^0)$.
\end{defn}

\begin{defn}\label{phi}
Let $(X=\Spec R, \D)$ be a pair and $e \ge 0$ be an integer.
Assume that $(p^e-1)(K_X+\D)$ is Cartier.
Then there exists an isomorphism
\[
\Hom_R(F^e_*( R((p^e-1)\D)) , R) \cong F^e_*R
\]
as $F^e_*R$-modules (see for example \cite[Lemma 3.1]{Sch09}).
We denote  by $\phi_\D^e$ a generator of $\Hom_R(F^e_*( R((p^e-1)\D)), R)$ as an $F^e_*R$-module.
\end{defn}

\begin{rem}
Although a map $\phi_\D^e : F^e_*R \to R$ is not uniquely determined, it is unique up to multiplication by $F^e_*R^\times$. When we consider this map, we only need the information about the image of this map.
Hence we ignore the multiplication by $F^e_* R^\times$.
\end{rem}

We now list some basic properties of test ideals.

\begin{lem}\label{test basic}
Let $(R,\D, \fa_\bullet^{t_\bullet}=\prod_{i=1}^m \fa_i^{t_i})$ be a triple.
Then the following hold.
\begin{enumerate}
\item[$(1)$] \textup{(\cite{Tak04}, cf. \cite[Proposition 3.23]{Sch11})} $\tau(R,\D, \fa_\bullet^{t_\bullet})=R$ if and only if the triple $(R,\D, \fa_\bullet^{t_\bullet})$ is strongly $F$-regular.
\item[$(2)$] Let $\D' \ge \D$ be another $\Q$-Weil divisor, $t'_i \ge t_i$ be another real numbers, and $\fa_i' \subseteq \fa_i$ be another ideals.
Then we have 
\[
\tau(R, \D', \prod_{i=1}^m (\fa')_i^{t'_i}) \subseteq \tau(R, \D , \fa_\bullet^{t_\bullet} ).
\]
\item[$(3)$] There exists a real number $\epsilon>0$ such that for every $t_i \le t_i' < t_i+\epsilon$, we have
\[
\tau(R, \D, \prod_{i=1}^m \fa_i^{t'_i}) = \tau(R, \D , \fa_\bullet^{t_\bullet} ).
\]

\item[$(4)$] \textup{(\cite[Theorem 4.2]{HT04}, cf. \cite[Lemma 3.26]{BSTZ10})} Let $\fb$ be an ideal generated by $l$ elements and $l \le s$ be a real number.
Then one has
\[
\tau(R, \D, \fa_\bullet^{t_\bullet} \fb^s)= \fb \tau(R, \D, \fa_{\bullet}^{t_{\bullet}} \fb^{s-1}).
\]
\item[$(5)$] \textup{(cf. \cite[Lemma 4.2, 4.4 (b)]{ST14b})} If $(p^e-1)(K_R+\D)$ is Cartier, then we have 
\[
\phi_{\D}^e ( F^e_* \tau( R, \D, \fa_{\bullet}^{t_\bullet}) ) = \tau(R, \D, \prod_{i=1}^m \fa_i^{t_i/p^e}).
\]
\item[$(6)$] \textup{(\cite[Proposition 3.23]{Sch11}, cf. \cite[Proposition 3.1]{HT04})} For any prime ideal $\p \subseteq R$, we have 
\[
\tau(R,\D,\fa_\bullet^{t_\bullet}) \cdot R_\p = \tau(R_\p, \D|_{R_\p}, \prod_{i=1}^m (\fa_i \cdot R_\p)^{t_i}),
\]
where $\D|_{R_{\p}}$ is the flat pullback of $\D$ to $\Spec R_\p$.
\item[$(7)$] \textup{(cf. \cite[Proposition 3.2]{HT04})} Let $\widehat{R}$ be the $\m$-adic completion. Then we have 
\[
\tau(R,\D,\fa_\bullet^{t_\bullet}) \cdot \widehat{R} = \tau(\widehat{R}, \D|_{\widehat{R}}, \prod_{i=1}^m (\fa_i \cdot \widehat{R})^{t_i}),
\]
where $\D|_{\widehat{R}}$ is the flat pullback of $\D$ to $\Spec \widehat{R}$.

\end{enumerate}
\end{lem}

\begin{proof}
The assertion in (3) follows from (2) and the ascending chain condition for ideals of $R$.
It follows from \cite[Lemma 3.21]{BSTZ10} that for all sufficiently large integer $n>0$, we have
\[
\begin{split}
\tau(R, \D, \fa_{\bullet}^{t_\bullet}) & = \phi^{en}(F^{en}_* \prod_i \fa_i^{\age{p^{en} t_i}} \tau(R,\D))  \textup{ and}\\
\tau(R, \D, \prod_{i=1}^m \fa_i^{t_i/p^e}) & = \phi^{en}(F^{en}_* \prod_i \fa_i^{\age{p^{e(n-1)} t_i}} \tau(R,\D)),
\end{split}
\]
which proves (5).
The assertion in (7) follows from \cite[Proposition 2.10 (iv)]{Sat19a}.
\end{proof}

\begin{lem}[\textup{\cite[Theorem 3.1]{Tak06}}]\label{sum}
Let $(R,\D)$ be a pair, $\fa, \fb \subseteq R$ be non-zero ideals and $t \ge 0$ be a real number.
Then we have 
\[
\tau(R, \D, (\fa+ \fb)^t)=\sum_{u,v \in \R_{\ge 0}, u+v=t} \tau(R,\D, \fa^u \fb^v)
\]
\end{lem}

\begin{lem}[\textup{\cite{ST14b}}]\label{discrete}
Let $(R,\D)$ be a pair such that $K_X+\D$ is $\Q$-Cartier, $\fa \subseteq R$ be an ideal and $t>0$ be a real number.
Then there exists a real number $\delta>0$ such that for all real numbers $t-\delta < t', t'' <t$, we have
\[
\tau(R,\D, \fa^{t'})=\tau(R, \D, \fa^{t''}).
\]
\end{lem}

\subsection{Base $q$ expansion}
\begin{defn}[\textup{cf. \cite[Definition 2.1, 2.2]{HnBWZ16}}]
Let $q \ge 2$ be an integer, $t>0$ be a real number and $n \in \Z$ be an integer.
We define the \emph{$n$-th digit} of $t$ in base $q$ by
\[
t^{(n)} : =\age{t q^n -1 } - q \age{t q^{n-1} -1} \in \Z.
\]
We define the \emph{$n$-th truncation} of $t$ in base $q$ by 
\[
\langle t \rangle _{n, q} : = \age{t q^n -1 }/q^n \in \Q.
\]
\end{defn}

\begin{lem}\label{qadic}
Let $q>1$ be an integer and $t>0$ be a real number.
Then the following hold.
\begin{enumerate}
\item[$(1)$] For any integer $n \in \Z$, we have 
\begin{enumerate}
\item[$($a$)$] $0 \le t^{(n)} <q$, 
\item[$($b$)$] $\qadic{t}{n} = \qadic{t}{n-1} + t^{(n)}/q^n$, and 
\item[$($c$)$] $t-(1/q^n) \le \qadic{t}{n} < t$.
\end{enumerate}
\item[$(2)$] If $t$ is a rational number, then after replacing $q$ by its power, $t^{(n)}$ is non-zero and constant for all $n \ge 2$.
\item[$(3)$] $t \in \Z[1/q]$ if and only if $t^{(n)}=q-1$ for all sufficiently large $n>0$.
\end{enumerate}
\end{lem}

\begin{proof}
The assertions in (1) and (3) follow easily from the definition.
For (2), replace $q$ by its multiple so that we have $q(q-1)t \in \Z$.
Then the assertion is obvious.
\end{proof}

\begin{lem}[\textup{\cite[Theorem 4.5, Convention 3.4]{HnBWZ16}}]\label{graded fpt}
Let $k$ be a perfect field of characteristic $p>0$ and $A=k[x,y]$ be a polynomial ring equipped with the structure of $\N$-graded ring such that $x$ and $y$ are homogeneous elements of positive degree.
Let $\m:=(x,y) \subseteq A$ be the homogeneous maximal ideal and $f \in A$ be a reduced homogeneous polynomial of degree $D>0$.
Assume that $D$ is coprime to $p$ and $\fpt(A_\m; (f)) \neq s : = \deg(xy)/D$.
Then there exists an integer $L \ge 1$ such that 
\[
\fpt(A_\m; (f))= \qadic[p]{s}{L}.
\]
\end{lem}

Let $R$ be a ring of characteristic $p>0$, $\fa \subseteq R$ be an ideal and $q=p^e$ be a power of $p$.
We denote by $\fa^{[q]}$ the ideal of $R$ generated by the set $\{f^q \mid f \in \fa\}$.
We note that if $\fa$ is generated by $f_1, \dots, f_n \in R$, then $\fa^{[q]}$ is generated by $f_1^q, \dots, f_n^q \in R$.

\begin{lem}[\textup{\cite[Corollary 2.30]{BMS08}}]\label{fpt qadic}
Let $k$ be an $F$-finite field of characteristic $p>0$, $A:=k[x_1, \dots, x_d]$ be a polynomial ring, $\m :=(x_1,\dots, x_d) \subseteq A$ be the maximal ideal and $f \in \m$ be a non-zero polynomial.
Then for every integer $e>0$, the $e$-th truncation $\qadic[p]{\fpt(A_\m; (f))}{e}$ of $\fpt(A_\m; (f))$ in base $p$ coincides with $\nu_f(p^e)/p^e$, where $\nu_f(p^e)$ is defined as
\[
\nu_f(p^e) := \sup\{n \ge 0 \mid f^n \not\in \m^{[p^e]} \}.
\]
\end{lem}

\subsection{Ultraproduct}
In this subsection, we define the ultraproduct of a family of sets and recall some properties.
We also define the catapower of a Noetherian local ring.
The reader is referred to \cite{Scho10} for details.

\begin{defn}
Let $\U$ be a collection of subsets of $\N$.
$\U$ is called an \emph{ultrafilter} if the following properties hold:
\begin{enumerate}
\item $\emptyset \not\in \U$.
\item For every subsets $A, B \subseteq \N$, if $A \in \U$ and $A \subseteq B$, then $B \in \U$.
\item For every subsets $A, B \subseteq \N$, if $A, B \in \U$, then $A \cap B \in \U$.
\item For every subset $A \subseteq \N$, if $A \not\in \U$, then $\N \setminus A \in \U$.
\end{enumerate}

An ultrafilter $\U$ is called \emph{non-principal} if the following holds:
\begin{enumerate}
\setcounter{enumi}{4}
\item If $A$ is a finite subset of $\N$, then $A \not\in \U$.
\end{enumerate}
\end{defn}

By Zorn's Lemma, there exists a non-principal ultrafilter.
From now on, we fix a non-principal ultrafilter $\U$.

\begin{defn}
Let $\{ T_m \}_{m \in \N}$ be a family of sets.
We define the equivalence relation $\sim$ on the set $\prod_{m \in \N} T_m$ by
\[
(a_m)_m \sim (b_m)_m \textup{ if and only if } 
\left\{ m \in \N \mid a_m=b_m \right\} \in \U.
\]
We define the \emph{ultraproduct} of $\{ T_m \}_{m \in \N}$ as
\[
\ulim_{m \in \N} T_m : = \left(\prod_{m \in \N} T_m \right) / \sim.
\]
If $T$ is a set and $T_m=T$ for all $m$, then we denote $\ulim_m T_m$ by $\ultra{T}$ and call it the \emph{ultrapower} of $T$.
\end{defn}

Let $\{ T_m \}_{m \in \N}$ be a family of sets and $a_m \in T_m$ for every $m$.
We denote by $\ulim_m a_m$ the class of $(a_m)_m$ in $\ulim_m T_m$.

Let $\{ R_m \}_{m \in \N}$ be a family of rings and $M_m$ be an $R_m$-module for every $m$.
Then $\ulim_m R_m$ has the ring structure induced by that of $\prod_m R_m$ and $\ulim_m M_m$ has the structure of $\ulim R_m$-module induced by the structure of $\prod_m R_m$-module on $\prod_m M_m$.
Moreover, if $k_m$ is a field for every $m$, then $\ulim_m k_m$ is a field.

Let $k$ be an $F$-finite field of positive characteristic.
Then the relative Frobenius morphism $F^e_* (k) \otimes_k \ultra{k} \to F^e_* (\ultra{k})$ is an isomorphism.
In particular, $\ultra{k}$ is an $F$-finite field. 

Let $(R, \m, k)$ be a local ring.
Then, one can show that $( \ultra{R}, \ultra{\m}, \ultra{k})$ is a local ring.
However, even if $R$ is Noetherian, the ultrapower $\ultra{R}$ may not be Noetherian because we do not have the equation $\cap_{n \in \N} (\ultra{\m})^n = 0$ in general.

\begin{defn}[\textup{\cite{Scho10}}]
Let $(R, \m)$ be a Noetherian local ring and $( \ultra{R}, \ultra{\m})$ be the ultrapower.
We define the \emph{catapower} $\cata{R}$ as the quotient ring
\[
\cata{R} : = \ultra{R}/ (\cap_{n} (\ultra{\m})^n).
\]
\end{defn}

\begin{prop}[\textup{\cite[Theorem 8.1.19]{Scho10}}]\label{cata cong}
Let $(R, \m, k)$ be a Noetherian local ring of equicharacteristic and $\widehat{R}$ be the $\m$-adic completion of $R$.
We fix a coefficient field $k \subseteq \widehat{R}$.
Then we have 
\[
\cata{R} \cong \widehat{R} \ \widehat{\otimes}_k (\ultra{k}).
\]
In particular, if $(R,\m)$ is an $F$-finite Noetherian normal local ring, then so is $\cata{R}$.
\end{prop}

Let $(R, \m)$ be a Noetherian local ring, $\cata{R}$ be the catapower and $a_m \in R$ for every $m$.
We denote by $\catae{a_m} \in \cata{R}$ the image of $\ulim_m a_m \in \ultra{R}$ by the natural projection $\ultra{R} \to \cata{R}$.
Let $\fa_m \subseteq R$ be an ideal for every $m \in \N$.
We denote by $[ \fa_m]_m \subseteq \cata{R}$ the image of the ideal $\ulim_m \fa_m \subseteq \ultra{R}$ by the projection $\ultra{R} \to \cata{R}$.

\begin{lem}[\textup{\cite[Proposition 2.10, Theorem 1.5, Proposition 4.5]{Sat19a}}]\label{cata fpt}
Let $(R,\D)$ be a strongly $F$-regular pair such that $(p^e-1)(K_R+\D)$ is Cartier for some integer $e>0$, $\{ \fa_m \}_{m \in \N}$ be a sequence of proper ideals of $R$ and $\cata{\D}$ be the flat pullback of $\D$ to $\Spec \cata{R}$.
Set $\fa_{\infty}: = [\fa_m]_m \subseteq \cata{R}$.
Then the following hold.
\begin{enumerate}
\item[$(1)$] $(\cata{R}, \cata{\D})$ is also a strongly $F$-regular pair and $(p^e-1)(K_{\cata{R}}+\cata{\D})$ is Cartier.
\item[$(2)$] If $t:= \lim_{m \to \infty} \fpt(R,\D;\fa_m)$ exists, then we have $t= \fpt(\cata{R}, \cata{\D}; \fa)$.
\item[$(3)$] If there exists an integer $N>0$ such that $\m^N \subseteq \fa_m$ for every $m$, then we have $\fpt(\cata{R}, \cata{\D}; \fa)= \fpt(R,\D; \fa)$ for infinitely many $m$.
\end{enumerate}
\end{lem}

\subsection{Sets of $F$-pure thresholds}

Let $(R,\D)$ be a sharply $F$-pure pair.
We denote by $\FPT(R,\D)$ the set of $F$-pure thresholds on $(R,\D)$, that is, 
\[
\FPT(R,\D) : = \{ \fpt(R,\D ; \fa) \mid \fa \subsetneq R \} \subseteq \R_{\ge 0}
\]
Similarly, we denote the set of $F$-pure thresholds of principal ideals on $(R,\D)$ by
\[
\FPT^{\mathrm{pr}}(R,\D) : = \{ \fpt(R,\D ; \fa) \mid \fa \subsetneq R \textup{ is a principal ideal} \} \subseteq \R_{\ge 0}
\]
For an integer $d \ge 1$, we define 
\[
\FPT_{<d}^{\mathrm{pr}}(R,\D) : = \bigcup_{\p \subseteq R} \FPT^{\mathrm{pr}}(R_{\p},\D|_{R_{\p}}),
\]
where $\p$ runs through all prime ideals of $R$ such that $1 \le \dim R_{\p} <d$ and $\D|_{R_{\p}}$ is the flat pullback of $\D$ by the natural morphism $\Spec R_{\p} \to \Spec R$.
When $\D=0$, we simply write $\FPT^{\mathrm{pr}}(R) := \FPT^{\mathrm{pr}}(R,0)$ and $\FPT^{\mathrm{pr}}_{<d}(R) := \FPT^{\mathrm{pr}}_{<d}(R,0)$.

\begin{defn}\label{TT}
Let $d>0$ be an integer, $p$ be a prime number and $k$ be an $F$-finite field.
\begin{enumerate}
\item[$(1)$] We denote by $P_d(k)$ the regular local ring defined as the localization of a polynomial ring $k[x_1, \dots, x_d]$ by the maximal ideal $(x_1, \dots, x_d) \subseteq k[x_1, \dots, x_d]$.
\item[$(2)$] We denote by $\TT_{d,p}^{\circ}$ the set of all $F$-pure thresholds on the regular local ring $P_d(\overline{\F_p})$, that is,
\[
\TT_{d,p}^{\circ} : = \FPT^{\mathrm{pr}}(P_{d}(\overline{\F_p})).
\]
\item[$(3)$] We denote by $\TT_{d,p}$ the set of all $F$-pure thresholds on all regular local rings of dimension $d$ and of characteristic $p$, that is,
\[
\TT_{d,p}:=\bigcup_A \FPT^{\mathrm{pr}}(A),
\]
where $A$ runs through all regular local ring of dimension $d$ and of characteristic $p$.
\end{enumerate}
\end{defn}

\begin{lem}\label{T basic}
Let $d>0$ be an integer, $p$ be a prime number.
Then the following holds.
\begin{enumerate}
\item[$(1)$] $\TT_{1,p}=\TT_{1,p}^{\circ} = \{0\} \cup \{1/m \mid m \in \N_{\ge 1} \}$.
\item[$(2)$] \textup{(\cite[Theorem 1.2]{BMS09})} The closure of the set $\TT_{d,p}^{\circ} \subseteq \R$ coincides with $\TT_{d,p}$.
\item[$(3)$] \textup{(\cite[Theorem 3.5]{BMS09})} For any algebraically closed field $k$ of characteristic $p$, we have $\TT_{d,p}^{\circ}=\FPT(P_d(k))$.
\item[$(4)$] There exists an $F$-finite filed $k$ of characteristic $p$ such that $\TT_{d,p}=\FPT(\widehat{P}_d(k))$, where $\widehat{P}_d(k)=k[[x_1, \dots,x_d]]$ is the ring of formal power series.
\item[$(5)$] If $d \ge 2$, then we have $\TT_{d-1,p}^{\circ} \subseteq \TT_{d,p}^{\circ}$ and $\TT_{d-1,p} \subseteq \TT_{d,p}$.
\end{enumerate}
\end{lem}

\begin{proof}
If $(A,\m)$ is a discrete valuation ring, then we have $\fpt(A;\m^n)=1/n$ and $\fpt(A;0)=0$, which proves (1).
For (4), let $k : = \ultra{(\overline{\F}_p)}$ be the ultrapower of the algebraic closure of $\F_p$.
Noting that $\TT_{d,p}$ is the closure of $\TT_{d,p}^{\circ}= \FPT^{\mathrm{pr}}(P_d(\overline{\F_p}))$, it follows from Lemma \ref{cata fpt} (2) that we have
\[
\TT_{d,p} \subseteq \FPT^{\mathrm{pr}} (\cata{(P_d(\overline{\F}_p))}).
\]
Then the assertion in (4) follows from the isomorphism $\cata{(P_d(\overline{\F}_p))} \cong \widehat{P}_d(k)$ which follows from Proposition \ref{cata cong}.

It follows from Lemma \ref{fpt qadic} that for an $F$-finite field $k$ and for a principal ideal $(f) \subsetneq P_d(k)$, we have
\[
\fpt(P_{d-1}(k); (f))= \fpt(P_d(k);(f)).
\]
which proves the inclusion $\TT_{d-1,p}^{\circ} \subseteq \TT_{d,p}^{\circ}$ in (5).
The inclusion $\TT_{d-1,p} \subseteq \TT_{d,p}$ is similar.
\end{proof}

\section{Construction of counterexamples}

\begin{prop}\label{making accumulation}
Let $A:=k[x_1, \dots, x_d]$ be a polynomial ring over an $F$-finite field $k$ equipped with a structure of $\N_{\ge 0}$-graded ring such that $x_i$ is a homogeneous element of positive degree $a_i>0$ for every $i$.
Let $\m : = (x_1, \dots, x_d)$ denote the homogeneous maximal ideal of $A$ and $f \in \m$ be a homogeneous polynomial of degree $D>0$.
Assume the following two conditions.
\begin{enumerate}
\item[$(a)$] $\fpt(A_\m; (f)) < \sum a_i/D$.
\item[$(b)$] $\fpt(A_\m; (f)) \not\in \Z[1/p]$.
\end{enumerate}
Then we have $\fpt(A_\m; (f)) \in \Acc(\FPT^{\mathrm{pr}}(A_\m))$, here for a subset $S$ of $\R$, we denote by $\Acc(S) \subseteq \R$ the set of all accumulation points of $S$.
\end{prop}

\begin{proof}
Set $t : =\fpt(A_\m; (f))$.
Let $E \subseteq \N_{\ge 1}$ be the set of all integers $e>0$ such that the $e$-th digit $t^{(e)}$ of $t$ in base $p$ is not $p-1$.
By Lemma \ref{qadic} (3) and the assumption (b), the set $E$ is an infinite set.
Fix an integer $e \in E$.

For any element $\lambda=(\lambda_1, \dots, \lambda_d) \in (\N_{\ge 0})^d$, we denote by $x^{\lambda}$ the monomial $\prod_{i=1}^d x_i^{\lambda_i}$.
We also define 
\[
I_e : =\{ \lambda =(\lambda_1, \dots, \lambda_d) \in (\N_{\ge 0})^d \mid  0 \le \lambda_i \le p^e-1\}.
\]
Consider the integer $\nu : = \nu_f(p^e)$ as in Lemma \ref{fpt qadic}.
We note that it follows from Lemma \ref{fpt qadic} that we have 
\[
\nu \equiv t^{(e)} (\mod p)
\]
and hence, $\nu +1$ is not divisible by $p$.
By the definition of $\nu_f(p^e)$, we have 
\[
f^{\nu} \not\in \m^{[p^e]}.
\]
Since $A/\m^{[p^e]}$ is a $k$-vector spaced based by $\{ x^{\lambda}\}_{\lambda \in I_e}$, there exists $\lambda=(\lambda_1, \dots, \lambda_d) \in I_e$ such that the coefficient of $f^\nu$ at $x^{\lambda}$ is non-zero.
We define the monomial $\alpha_e$ by
\[
\alpha_e : = \prod_{i=1}^d x_i^{p^e-1-\lambda_i} \in A.
\]
Since $\deg (x^\lambda)=\deg(f^{\nu}) = \nu D$, it follows from Lemma \ref{fpt qadic} that the degree of $\alpha_e$ satisfies
\begin{eqnarray*}
\deg(\alpha_e) &=& p^e(\sum_{i=1}^d a_i -\frac{\nu D}{p^e}) - \sum_{i=1}^d a_i \\
&>& p^e( \sum_{i=1}^d a_i -t D) - \sum_{i=1}^d a_i.
\end{eqnarray*}
In particular, by the assumption (a), we have 
\[
\lim_{E \ni e \to \infty} \deg(\alpha_e) = \infty.
\]

Let $E' \subseteq E$ be the subset consists of all $e \in E$ such that $\deg(\alpha_e) \neq D$.
We note that $E'$ is an infinite set since $\lim_{e \to \infty} \deg(\alpha_e)=\infty$.
Fix an integer $e \in E'$ and consider $\nu = \nu_f(p^e)$.
Set $g_e : = f +\alpha_e \in A$.
Since we have
\[
\lim_{E' \ni e \to \infty} \deg(g_e-f) =\lim_{E' \ni e \to \infty} \deg (\alpha_e) =\infty,
\]
it follows from \cite[Corollary 3.4]{BMS09} that 
\[
\lim_{E' \ni e \to \infty} \fpt(A_\m; (g_e))=t.
\]

On the other hand, since $\deg(f) \neq \deg(\alpha_e)$, the decomposition
\[
g_e^{\nu+1} = \sum_{i=0}^{\nu +1} \binom{\nu+1}{i} f^{\nu+1-i} \alpha_e^i
\]
is the homogeneous decomposition, where $\binom{\nu+1}{i}$ is the binomial coefficient.
Since a homogeneous component $\binom{\nu+1}{1} f^{\nu} \alpha_e$ is not contained in the homogeneous ideal $\m^{[p^e]}$, nor is $g_e^{\nu+1}$.
Therefore, we have
\[
\nu_{g_e}(p^e) \ge \nu_f(p^e) +1.
\]
It follows from Lemma \ref{fpt qadic} that we have
\[
t < \fpt(A_\m; (g_e)),
\]
which proves that $t$ is an accumulation point of $\{\fpt(A_\m; (g_e))\}_{e \in E'}$.
\end{proof}

\begin{cor}[Theorem \ref{1.3}]\label{Counter ex}
For any prime number $p$, we have
\[
\Acc(\TT_{2,p}) \not\subseteq \TT_{1,p}.
\]
\end{cor}

\begin{proof}
We first consider the case where $p \ge 5$.
We define $a_p : = (p-1)/2 \in \Z$.
Let $k$ be an algebraic closure $\overline{\F_p}$ of $\F_p$, $R : =k[z]$ be a polynomial ring and $B:=R[x,y]$ be a polynomial ring over $R$ with two variables $x,y$.
Consider a polynomial
\[
G(x,y,z):= xy(x+y)(x+zy) \in B.
\]
Let $H(z) \in R$ be the coefficient of the power $G(x,y,z)^{a_p} \in R[x,y]$ at the monomial $x^{p-1}y^{p-1}$.
Since $H(z) \in k[z]$ is a non-constant monic polynomial and $k$ is algebraically closed, there exists an element $u \in k$ such that $H(u)=0$.
We note that $u$ is not $0$ because the coefficient of the polynomial $G(x,y,0)^{a_p}$ at the monomial $x^{p-1}y^{p-1}$ is $1$.
Similarly, we have $u \neq 1$.

Let $A:=k[x,y]$ be a polynomial ring equipped with a structure of $\N_{\ge 0}$-graded ring such that $x$ and $y$ are homogeneous of degree $1$ and $\m:=(x,y) \subseteq A$ be the homogeneous maximal ideal.
We consider a homogeneous polynomial 
\[
f(x,y) := G(x,y,u) =xy(x+y)(x+uy) \in A.
\]
Then it follows from the definition of $u$ that 
\[
\nu_f(p) < a_p.
\]
By Lemma \ref{fpt qadic} and \ref{qadic} (1), we have
\[
\fpt(A_\m; (f)) \le \frac{a_p}{p}=\qadic[p]{\frac{1}{2}}{1}.
\]
Therefore, Lemma \ref{graded fpt} implies that 
\[
\fpt(A_\m; (f)) = \frac{a_p}{p}
\]

Take an integer $m>0$ such that $m$ is not a power of $p$ and is coprime to the integer $a_p$.
Then the homogeneous polynomial $g : =f^m$ satisfies the conditions in Proposition \ref{making accumulation}.
Therefore, we have
\[
\fpt(A_\m; (g)) =\frac{a_p}{pm} \in \Acc(\TT_{2,p}).
\]
On the other hand, since $m$ is coprime to $a_p$, one has
\[
\frac{a_p}{pm} \not\in \TT_{1,p} =\{0\} \cup \{1/n \mid n \in \N_{\ge 1} \},
\]
which proves the corollary when $p \ge 5$.

In the case where $p=3$, we consider a polynomial
\[
f(x,y) : = xy(x+y)(x-y)(x^3-xy^2-y^3) \in \F_3[x,y] = :A,
\]
where we set $\deg(x)=\deg(y)=1$.
We see that $f^2 \in \m^{[9]}$, which implies $\fpt(A_\m;(f)) \le 2/9$.
Then it follows from Lemma \ref{graded fpt} that $\fpt(A_\m; (f))=2/9$.
In particular, we have
\[
\fpt(A_\m; (f^5)) = \frac{2}{45} \in \Acc(\TT_{2,3}) \setminus \TT_{1,3}.
\]

When $p=2$, we consider a polynomial
\[
f(x,y) : = xy(x^4+x^2y +y^2)  \in \F_2[x,y] =: A,
\]
where we set $\deg(x)=1$ and $\deg(y)=2$.
Then we see that 
\[
\fpt(A_\m; (f^5)) = \frac{3}{40} \in \Acc(\TT_{2,2}) \setminus \TT_{1,2}.
\]
\end{proof}

\section{Positive results}
\begin{defn}[cf. \cite{Per13}]
Let $(R,\D)$ be a strongly $F$-regular pair and $\fa_1, \dots, \fa_l$ be non-zero ideals.
We define the \emph{strongly $F$-regular region} of $\fa_1, \dots, \fa_l$ with respect to $(R,\D)$ by
\[
\begin{split}
\SFRR(R,\D; \fa_1, \dots, \fa_l)  : = & \left\{ (t_1, \dots, t_l) \in \R_{\ge 0}^l \ \middle|\  \begin{array}{l}  (R, \D, \fa_1^{t_1} \cdots \fa_l^{t_l}) \textup{ is strongly} \\ \textup{$F$-regular} \end{array} \right\} \\
\subseteq & \ \R^l_{\ge 0} .
\end{split}
\]
\end{defn}

\begin{rem}\label{SFRR and fpt}
Let $(R,\D)$ be a strongly $F$-regular pair, $\fa_1, \dots, \fa_l \subsetneq R$ be non-zero proper ideals and $\lambda_1, \dots, \lambda_l \ge 0$ be integers.
Set $\fb : = \prod_{i=1}^l \fa_i^{\lambda_i}$ and $\lambda : =(\lambda_1, \dots, \lambda_l) \in \Z^l$.
Then we have
\[
\fpt(R,\D;\fb) = \sup\left\{ t \ge 0 \mid t \cdot \lambda \in \SFRR(R,\D; \fa_1, \dots, \fa_l) \right\}.
\]
\end{rem}

\begin{defn}
Let $(R,\D)$ be a strongly $F$-regular pair, $\m \subseteq R$ be the maximal ideal, and $\fa \subsetneq R$ be a proper ideal.
We say that $\fa$ is \emph{of vertical type} with respect to $(R,\D)$ if there exists a real number $\delta >0$ such that 
\[
\fpt(R, \D, \m^{\delta}; \fa)=\fpt(R, \D; \fa).
\]
\end{defn}

\begin{prop}\label{vertical basic}
Let $(R,\D)$ be a strongly $F$-regular pair such that $(p^e-1)(K_R+\D)$ is Cartier for some integer $e>0$ and $\fa \subsetneq R$ be a non-zero proper ideal.
Set $t_0 : =\fpt(R,\D;\fa)$.
Then the following are equivalent.
\begin{enumerate}
\item[$(\textup{i})$] $\fa$ is of vertical type with respect to $(R,\D)$.
\item[$(\textup{ii})$] $\SFRR(R,\D; \fa, \m) \cap \ell \neq \emptyset$ for any line $\ell \subseteq \R^2$ such that $(t_0,0) \in \ell$ and the slope of $\ell$ is negative.
\item[$(\textup{iii})$] $\SFRR(R,\D; \fa, \m) \cap \ell_n \neq \emptyset$ for all integer $n>0$, where $\ell_n \subseteq \R^2$ is the line such that $(t_0,0), (0,n) \in \ell_n$.
\item[$(\textup{iv})$] $\fpt(R,\D; \fa+ \m^n) > t_0$ for all integers $n>0$.
\end{enumerate}
\end{prop}

\begin{proof}
The equivalence (ii) $\Leftrightarrow$ (iii) is obvious.
If there exists a real number $\delta>0$ such that $\fpt(R, \D, \m^{\delta}; \fa)= t_0$ and $\delta< \fpt(R,\D; \m)$, then it follows from Lemma \ref{F-sing basic} that for every real number $0<t<t_0$, we have $(t, \delta) \in \SFRR(R,\D; \fa, \m)$, which implies the implication (i) $\Rightarrow$ (ii).

It follows from Lemma \ref{F-sing basic} and Lemma \ref{test basic} (1) that the assertion in (iv) is equivalent to the equations $\tau(R,\D, (\fa+\m^n)^{t_0})=R$ for all $n>0$.
By Lemma \ref{sum}, we have 
\[
\tau(R,\D, (\fa+\m^n)^{t_0}) = \sum_{(t,s) \in \ell_n \cap \R_{\ge0}^2 } \tau(R,\D,\fa^t \m^s),
\]
which implies the equivalence (iv) $\Leftrightarrow$ (iii).

For the implication (ii) $\Rightarrow$ (i), we assume that $\fa$ is not of vertical type and we will prove that there exists a line $\ell \subseteq \R^2$ such that  $(t,0) \in \ell$, the slope is negative and $\SFRR(R,\D;\fa, \m) \cap \ell = \emptyset$.

Let $0 \le t_0^{(n)} \le q-1$ be the $n$-th digit of $t_0$ in base $q : = p^e$.
Noting that $t_0$ is a rational number (\cite[Theorem B]{ST14b}), it follows from Lemma \ref{qadic} (2) that after replacing $e$ by its multiple, we may assume that $t_0^{(n)}$ is non-zero and constant for $n \ge 2$ .
Set $a:= t_0^{(2)}$.
For any integer $n \ge 0$, we set $x_n : = \qadic{t_0}{n} $.
By Lemma \ref{qadic} (1), we have 
\[
x_{n+1}=x_n + a/q^{n+1}
\]
for every $n \ge 1$ and $\lim_{n \to \infty} x_n=t_0$.

Set $y_n : = \fpt(R, \D, \fa^{x_n}; \m)$ for every integer $n \ge 1$.
Since $\fa$ is not of vertical type, we can see that $\lim_{n \to \infty} y_n =0$.
On the other hand, by \cite[Corollary 5.5]{Sat19a}, there exists an integer $N>0$ such that
\[
y_{n+1} \ge y_n -N/q^n
\]
for any integer $n \ge 0$.

Let $\ell' \subseteq \R^2$ be the line such that $(t_0,0) \in \ell'$ and the slope is $-qN/a$, that is, 
\[
\ell' : =\left\{ (x,y) \in \R^2 \mid y=- \frac{qN}{a} (x-t_0) \right\}. 
\]
Let $y_n'$ be the real number such that $(x_n, y_n') \in \ell'$.
Then for every $n \ge 1$, we can see that $y_n \le y_n'$, which implies $(x_n, y_n') \not\in \SFRR(R,\D; \fa, \m)$.

Let $\ell''$ be the line which passes through $(t_0,0)$ and $(x_n, y_{n+1}')$ for all $n \ge 1$, that is,
\[
\ell'' : =\left\{ (x,y) \in \R^2 \mid y=- \frac{q^2N}{a} (x-t_0) \right\}. 
\]
Then for any point $(x, y) \in \ell'' \cap \R_{\ge 0}^2$ with $x \ge x_1$, there exists $n \ge 1$ such that $x \ge x_n$ and $y \ge y_n'$.
Since $(x_n, y'_n)$ is not contained in $\SFRR(R,\D;\fa,\m)$, nor is $(x,y)$.

Finally, let $M>q^2N/a$ be a real number and $\ell$ be the line such that $(t_0, 0) \in \ell$ and the slope is $-M$. 
Since the slope $-M$ of $\ell$ is smaller than that of $\ell''$, for any point $(x,y) \in \ell$ with $x \ge x_1$, we have $(x,y) \not\in \SFRR(R,\D;\fa, \m)$.
On the other hand, if $M$ is sufficiently large, then for every point $(x,y) \in \ell$ with $x<x_1$, we have $y \ge \fpt(R,\D;\m)$, which implies $(x,y) \not\in \SFRR(R,\D;\fa,\m)$.
Therefore we have $\SFRR(R,\D;\fa, \m) \cap \ell = \emptyset$, as desired.
\end{proof}

\begin{thm}\label{accum characterize}
Let $(R,\D)$ be a strongly $F$-regular pair such that $(p^e-1)(K_R+\D)$ is Cartier for some integer $e>0$, $t>0$ be a real number and $\cata{\D}$ be the flat pullback of $\D$ to the catapower $\cata{R}$.
Then the following hold.
\begin{enumerate}
\item[$(1)$] $t$ is an accumulation point of the set $\FPT(R,\D)$ if and only if there exists a non-zero ideal $\fa \subsetneq \cata{R}$ such that $t= \fpt(\cata{R}, \cata{\D}; \fa)$ and $\fa$ is of vertical type with respect to $(\cata{R}, \cata{\D})$.
\item[$(2)$] If $t$ is an accumulation point of the set $\FPT^{\mathrm{pr}}(R,\D)$, then there exists a non-zero principal ideal $\fa \subsetneq \cata{R}$ such that $t= \fpt(\cata{R}, \cata{\D}; \fa)$ and $\fa$ is of vertical type with respect to $(\cata{R}, \cata{\D})$.
\end{enumerate}
\end{thm}

\begin{proof}
We first prove the assertion in (2).
We assume that $t$ is an accumulation point of $\FPT^{\mathrm{pr}} (R,\D)$.
Then there exists a sequence of proper principal ideals $\{\fa_m=(f_m)\}_{m \in \N} \subsetneq R$ such that $t_m : = \fpt(R, \D; \fa_m)$ satisfies $t_m \neq t$ for every $m \in \N$ and $\lim_{m \to \infty} t_m=t$. 
Since the set $\FPT^{\mathrm{pr}}(R,\D)$ satisfies the ascending chain condition (\cite{Sat19a}), after replacing by a subsequence, we may assume that $t_m > t$ for every $m$.

Set $\fa : = [ \fa_m ]_m \subseteq \cata{R}$.
We note that $\fa$ is a principal ideal generated by $\catae{f_m} \in \cata{R}$.
By Lemma \ref{cata fpt} (2), we have $t=\fpt(\cata{R}, \cata{\D}; \fa)$.
Let $\cata{\m} \subseteq \cata{R}$ be the maximal ideal.
For every integer $n>0$, since we have $\fa+ \cata{\m}^n = [ \fa_m+\m^n]_m$, it follows from Lemma \ref{cata fpt} (3) that there exists $m \in \N$ such that 
\[
\fpt(\cata{R}, \cata{\D}; \fa+\cata{\m}^n) = \fpt(R, \D; \fa_m+\m^n) \ge \fpt(R,\D; \fa_m) > t .
\]
Therefore, it follows from Proposition \ref{vertical basic} that $\fa$ is of vertical type, which complete the proof of (2).
The proof of the only if part of (1) is similar.

Next, we consider the if part of (1).
We assume that there exists a non-zero proper ideal $\fa \subseteq \cata{R}$ such that $t = \fpt(\cata{R}, \cata{\D}; \fa)$ and $\fa$ is of vertical type.
Then it follows from Proposition \ref{vertical basic} that $t_n : = \fpt(\cata{R}, \cata{\D}; \fa+\cata{\m}^n) >t$ for every $n>0$. 
By considering a generator of $\fa$, we can construct a sequence of non-zero ideals $\{ \fa_m \}_m$ such that $\fa = \catae{\fa_m }$.
By Lemma \ref{cata fpt} (3), for every $n>0$, there exists an integer $m_n \in \N$ such that  
\[
t_n=\fpt(R,\D; \fa_{m_n}+\m^n) \in \FPT(R,\D).
\]
Since $t = \lim_{n \to \infty} t_n$, we conclude that $t$ is an accumulation point of the set $\FPT(R,\D)$.
\end{proof}



\begin{lem}\label{isolated non-vert}
Let $(R,\D)$ be a strongly $F$-regular pair such that $(p^e-1)(K_R+\D)$ is Cartier for some integer $e>0$, $\fa$ be a non-zero proper principal ideal and $t : = \fpt(R,\D; \fa)$.
If the test ideal $\tau(R,\D, \fa^t)$ is $\m$-primary and the denominator of $t$ is not divisible by $p$, then $\fa$ is not of vertical type with respect to $(R,\D)$.
\end{lem}

\begin{proof}
Let $\phi =\phi^e_\D : F^e_* R \to R$ be as in Definition \ref{phi}.
Take an integer $M>0$ such that $\m^M \subseteq \tau(R,\D, \fa^t)$.
Let $r:=\dim(\m/\m^2)$ be the embedding dimension of $R$.
After replacing $e$ by its multiple, we may assume that $t(p^e-1)$ is an integer.

Noting that $\tau(R,\D; \fa^t) \subseteq \m$ and the test ideal is uniformly $(\D,\fa^t, F)$-compatible, we have 
\begin{equation}\label{1}
\phi^n( F^{en}_* (\fa^{t(p^{en}-1)} \cdot \tau(R,\D, \fa^t))) \subseteq \tau(R,\D, \fa^t) \subseteq \m
\end{equation}
for every $n>0$.
On the other hand, since $\fa$ is principal, it follows from Lemma \ref{test basic} (4) that 
\begin{eqnarray*}
\fa^{t(p^{en}-1)} \cdot \tau(R,\D, \fa^t) \supseteq \fa^{t(p^{en}-1)} \cdot \m^M &\supseteq& \fa^{t(p^{en}-1)} \cdot \m^M \cdot \tau(R,\D, \m^r) \\
&\supseteq& \tau(R,\D, \fa^{t(p^{en}-1)} \m^{r+M}).
\end{eqnarray*}
Combining with the inclusion \eqref{1} and Lemma \ref{test basic} (5), we have
\[
\tau(R,\D, \fa^{t(p^{en}-1)/p^{en}} \m^{(M+r)/p^{en}}) \subseteq \m,
\]
which implies $\fpt(R,\D, \m^{(M+r)/p^{en}} ;\fa) <t(p^{en}-1)/ p^{en} <t$ for every $n>0$.
Since $n$ runs through all positive integers, we have $\fpt(R, \D, \m^\delta; \fa) <t$ for every $\delta>0$, as desired.
\end{proof}

\begin{thm}\label{accum on pair}
Let $(R,\D)$ be a $d$-dimensional strongly $F$-regular pair such that $(p^e-1)(K_R+\D)$ is Cartier for some integer $e>0$.
Then we have
\[
\Acc(\FPT^{\mathrm{pr}}(R, \D)) \cap \Z_{(p)} \subseteq \FPT_{<d}^{\mathrm{pr}}(\cata{R}, \cata{\Delta}).
\]
\end{thm}

\begin{proof}
Take $t \in \Acc(\FPT^{\mathrm{pr}}(R, \D)) \cap \Z_{(p)}$.
By Theorem \ref{accum characterize} (2), there exists a principal ideal $\fa \subsetneq \cata{R}$ such that $t= \fpt(\cata{R}, \cata{\D}; \fa)$ and $\fa$ is of vertical type with respect to $(\cata{R}, \cata{\D})$.
Applying Lemma \ref{cata fpt} (1) and Lemma \ref{isolated non-vert}, the ideal $\tau(R,\D, \fa^t)$ is not $\cata{m}$-primary, where $\cata{\m} \subseteq \cata{R}$ is the maximal ideal.
Take a prime ideal $\p \neq \cata{\m} \subseteq \cata{R}$ such that $\tau(\cata{R},\cata{\D}, \fa^t)_{\p} \neq (\cata{R})_{\p}$.
Since test ideals are compatible with the localizations (Lemma \ref{test basic} (6)), we have
\[
t= \fpt((\cata{R})_{\p}, \cata{\D}|_{(\cata{R})_{\p}}; \fa_{\p}).\] 
On the other hand, since $\dim \cata{R} =d$ by Lemma \ref{cata cong}, we have $\dim (\cata{R})_{\p} <d$ and hence we conclude that 
\[
t \in \FPT^{\mathrm{pr}}_{<d}(\cata{R}, \cata{\D}).
\]
\end{proof}

\begin{cor}[Theorem \ref{1.4}]\label{Accum positive}
Let $p$ be a prime number and $d \ge 2$ be an integer.
Then we have 
\[
\Acc(\TT_{d,p}) \cap \Z_{(p)} \subseteq \TT_{d-1,p}.
\]
\end{cor}

\begin{proof}
Take an $F$-finite field $k$ of characteristic $p$ such that $A=\widehat{P}_d(k)$ satisfies $\TT_{d,p} =\FPT^{\mathrm{pr}}(A)$ (Lemma \ref{T basic} (4)).
Then it follows from Theorem \ref{accum on pair} that 
\[
\Acc(\TT_{d,p}) \cap \Z_{(p)} = \Acc(\FPT^{\mathrm{pr}}(A)) \cap \Z_{(p)} \subseteq  \FPT^{\mathrm{pr}}_{<d}(\cata{A}).
\]
Noting that $\cata{A}$ is a regular local ring by Proposition \ref{cata cong}, we have
\[
\FPT^{\mathrm{pr}}_{<d} (\cata{A}) \subseteq \TT_{d-1, p},
\]
which completes the proof.
\end{proof}

\begin{cor}[Corollary \ref{1.6}]\label{fpt polynomial}
Let $p$ be a prime number and $d \ge 1$ be an integer.
Then we have 
\[
\TT_{d,p} \cap \Z_{(p)} \subseteq \TT^{\circ}_{d,p}.
\]
\end{cor}

\begin{proof}
When $d=1$, the assertion follows from Lemma \ref{T basic} (1).
We assume that $d \ge 2$ and the assertion holds in $d-1$.
It follows from Lemma \ref{T basic} (2) that 
\[
\TT_{d,p}=\Acc(\TT_{d,p}^{\circ}) \cup \TT_{d,p}^{\circ} \subseteq \Acc(\TT_{d,p}) \cup \TT_{d,p}^{\circ}.
\]
By Corollary \ref{Accum positive}, we have
\[
\TT_{d,p} \cap \Z_{(p)} \subseteq (\TT_{d-1,p} \cap \Z_{(p)}) \cup \TT_{d,p}^{\circ}.
\]
By induction hypothesis and Lemma \ref{T basic} (5), we have 
\[
\TT_{d-1,p} \cap \Z_{(p)} \subseteq \TT_{d-1,p}^{\circ} \subseteq \TT_{d,p}^{\circ},
\]
which completes the proof.
\end{proof}

\section{Two dimensional case}

\begin{prop}\label{stability of region} 
Let $(R,\D)$ be a strongly $F$-regular pair such that $(p^e-1)(K_R+\D)$ is Cartier for some integer $e>0$ and $f_1, \dots, f_l \in \m$ be non-zero elements such that the triple $(R, \D, \prod_{i=1}^l (f_i)^1)$ is sharply $F$-pure outside $\m$.
Then there exists an integer $M>0$ such that for any real number $0 \le t_i <1$ and any element $g_i \in R$ such that $g_i \equiv f_i (\mod \m^M)$ for $i=1, \dots, l$, we have 
\begin{equation}\label{same tau}
\tau(R, \D, \prod_{i=1}^l (f_i)^{t_i})=\tau(R, \D, \prod_{i=1}^l (g_i)^{t_i}).
\end{equation}
In particular, we have $\SFRR(R,\D; (f_1), \dots, (f_l)) =\SFRR(R,\D; (g_1), \dots, (g_l))$.
\end{prop}

\begin{proof}
Set $f := \prod_{i=1}^l f_i \in R$ and $q:=p^e$.
It follows from Lemma \ref{test basic} (3) and Lemma \ref{discrete} that there exists an ideal $\fq \subseteq R$ such that $\tau(R,\D, (f)^{1-\epsilon})=\fq$ for any sufficiently small $\epsilon >0$.
We note that by Lemma \ref{test basic} (2), for any real number $0 \le t_i <1$, we have 
\[
\fq \subseteq \tau(R,\D, \prod_{i=1}^l (f_i)^{t_i}).
\]
On the other hand, it follows from Lemma \ref{F-sing basic} and Lemma \ref{test basic} (1) (6) that the ideal $\fq$ is $\m$-primary.
Let $M>0$ be an integer such that 
\[
\m^M \subseteq (\m \fq)^{[q]}.
\]

Take $g_1, \dots, g_l \in R$ such that $g_i \equiv f_i \mod \m^M$.
Noting that $\Z[1/p] \subseteq \R$ is a dense subset, combining with Lemma \ref{test basic} (3), it is enough to verify the equation \eqref{same tau} in the assertion in the case where $t_i \in [0,1) \cap \Z[1/q]$ for all $i$.

Fix $t_1, \dots, t_l \in [0,1) \cap \Z[1/p]$ and Let $n \ge 0$ be the smallest integer such that $q^n t_i \in \Z$ for all $i$.
If $n=0$, then we have $t_1= \dots=t_l=0$ and hence the equation \eqref{same tau} is obvious.

We assume that $n \ge 1$ and the assertion holds for $n-1$.
Let 
\[
qt_i = \sage{qt_i} + \fr{qt_i}
\] 
be the decomposition of $qt_i$ into the integral part $\sage{qt_i} \in \Z$ and the fractional part $\fr{qt_i} \in [0,1)$.
Consider the triple $(R,\D, f_{\bullet}^{\fr{qt}_\bullet}: = \prod_{i=1}^l f_i^{\fr{qt_i}})$ and set 
\[
I : = \tau(R, \D, f_\bullet^{\fr{qt_\bullet}}) \subseteq R.
\]
Then it follows from induction on $n$ that we have $I= \tau(R, \D, g_\bullet^{\fr{qt_\bullet}})$,
where $g_\bullet^{\fr{qt_\bullet}} = \prod_{i=1}^l g_i^{\fr{qt_i}}$.
It follows from Lemma \ref{test basic} (4) and (5) that
\begin{equation}\label{tau henkei}
\tau(R,\D, f_\bullet^{t_\bullet})=\phi^e_{\D}(F^e_* \tau(R,\D, f_\bullet^{qt_{\bullet}}))=\phi^e_\D(F^e_* f_\bullet^{\sage{qt_\bullet}} I),
\end{equation}
where we define $f_\bullet^{qt_\bullet} : = \prod_i f_i^{qt_i}$ and $f_\bullet^{\sage{qt_\bullet}} : = \prod_i f_i^{\sage{qt_i}} \in R$.
Similarly, we have
\begin{equation}\label{tau henkei2}
\tau(R,\D, g_\bullet^{t_\bullet})=\phi^e_\D(F^e_* g_\bullet^{\sage{qt_\bullet}} I),
\end{equation}
where $g_\bullet^{t_\bullet} : = \prod_i g_i^{t_i}$ and $g_\bullet^{\sage{qt_\bullet}}:=\prod_i g_i^{\sage{qt_i}} \in R$.

On the other hand, since $(f_i)+\m^M=(g_i)+\m^M$ for all $i$, we have $f_\bullet^{\sage{qt_\bullet}} I +\m^M = g_\bullet^{\sage{qt_\bullet}} I + \m^M$.
Combining with the equations \eqref{tau henkei} and \eqref{tau henkei2}, we have
\begin{equation}\label{before NAK}
\tau(R,\D, f_\bullet^{t_\bullet}) + \phi^e_\D(F^e_* \m^M)=\tau(R,\D,g_\bullet^{t_\bullet}) +\phi^e_\D(F^e_* \m^M).
\end{equation}

It follows from the definition of $M$ that one has
\[
\phi^e_\D(F^e_* \m^M) \subseteq \phi^e_\D(F^e_* ((\m \fq)^{[q]})) =\m \fq \phi^e_\D(F^e_* R) \subseteq \m \fq \subseteq \m \tau(R,\D, f_\bullet^{t_\bullet}).
\]
Similarly, we have
\[
\phi^e_\D(F^e_* \m^M) \subseteq \m \tau(R,\D, g_\bullet^{t_\bullet}).
\]
Then the equation \eqref{same tau} in the assertion follows from the equation \eqref{before NAK} by Nakayama.
\end{proof}

\begin{cor}[Theorem \ref{1.7}]\label{two-dim}
For any prime number $p>0$, we have
\[
\TT_{2,p} = \TT_{2,p}^{\circ}.
\]
\end{cor}

\begin{proof}
Take an algebraically closed field $k$ of characteristic $p$ such that the regular local ring $A:=P_2(k)$ satisfies $\TT_{2,p}=\FPT^{\mathrm{pr}}(\widehat{A})$ (Lemma \ref{T basic} (4)).
Take a non-zero principal ideal $(f) \subsetneq \widehat{A}$.
By lemma \ref{T basic} (3), it is enough to show that there exists a non-zero principal ideal $(g) \subseteq A$ such that $\fpt(\widehat{A};(f)) = \fpt(A; (g))$.
Consider the irreducible decomposition $f= \prod_{i=1}^l f_i^{n_i} \in \widehat{A}$.
Since the triple $(\widehat{A}, 0 , \prod_{i=1}^l (f_i)^1)$ is sharply $F$-pure in codimension one, applying Proposition \ref{stability of region}, there exists $g_1, \dots, g_l \in A'$ such that 
\[\SFRR(\widehat{A},0; (f_1), \dots, (f_l)) =\SFRR(\widehat{A},0; (g_1), \dots, (g_l)).\]
Set $g:= \prod_{i=1}^l g_i^{n_i}$.
Then it follows from Remark \ref{SFRR and fpt}, Lemma \ref{test basic} (7) and Lemma \ref{T basic} (3) that 
\[
\fpt(\widehat{A};(f))=\fpt(\widehat{A}; (g)) = \fpt(A; (g))
\]
as desired.
\end{proof}

\bibliographystyle{abbrv}
\bibliography{ref}

\begin{thebibliography}{10}

\bibitem{Bir07}
C.~Birkar.
\newblock \textup{Ascending chain condition for log canonical thresholds and
  termination of log flips}.
\newblock {\em Duke Mathematical Journal}, 136(1):173--180, 2007.

\bibitem{Bir16}
C.~Birkar.
\newblock \textup{Singularities of linear systems and boundedness of Fano
  varieties}.
\newblock {\em arXiv preprint arXiv:1609.05543}, 2016.

\bibitem{BMS08}
M.~Blickle, M.~Musta\c{t}\u{a}, and K.~E. Smith.
\newblock \textup{Discreteness and rationality of $F$-thresholds}.
\newblock {\em The Michigan mathematical journal}, 57(1):43--61, 2008.

\bibitem{BMS09}
M.~Blickle, M.~Musta\c{t}\u{a}, and K.~E. Smith.
\newblock \textup{$F$-thresholds of hypersurfaces}.
\newblock {\em Transactions of the American Mathematical Society},
  361(12):6549--6565, 2009.

\bibitem{BSTZ10}
M.~Blickle, K.~Schwede, S.~Takagi, and W.~Zhang.
\newblock \textup{Discreteness and rationality of $F$-jumping numbers on
  singular varieties}.
\newblock {\em Mathematische Annalen}, 347(4):917--949, 2010.

\bibitem{dFEM10}
T.~de~Fernex, L.~Ein, and M.~Musta{\c{t}}{\u{a}}.
\newblock \textup{Shokurov's ACC conjecture for log canonical thresholds on
  smooth varieties}.
\newblock {\em Duke Mathematical Journal}, 152(1):93--114, 2010.

\bibitem{HMX14}
C.~D. Hacon, J.~M\textsuperscript{c}Kernan, and C.~Xu.
\newblock \textup{ACC for log canonical thresholds}.
\newblock {\em Annals of Mathematics}, pages 523--571, 2014.

\bibitem{HT04}
N.~Hara and S.~Takagi.
\newblock \textup{On a generalization of test ideals}.
\newblock {\em Nagoya Mathematical Journal}, 175:59--74, 2004.

\bibitem{HnBWZ16}
D.~J. Hern{\'a}ndez, L.~N{\'u}{\~n}ez-Betancourt, E.~E. Witt, and W.~Zhang.
\newblock \textup{$F$-Pure thresholds of homogeneous polynomials}.
\newblock {\em The Michigan Mathematical Journal}, 65(1):57--87, 2016.

\bibitem{Kol97}
J.~Koll{\'a}r.
\newblock \textup{Singularities of pairs}.
\newblock In {\em Proceedings of Symposia in Pure Mathematics}, volume~62,
  pages 221--288. American Mathematical Society, 1997.

\bibitem{Kun76}
E.~Kunz.
\newblock \textup{On Noetherian rings of characteristic $p$}.
\newblock {\em American Journal of Mathematics}, 98(4):999--1013, 1976.

\bibitem{Per13}
F.~P{\'e}rez.
\newblock \textup{On the constancy regions for mixed test ideals}.
\newblock {\em Journal of Algebra}, 396:82--97, 2013.

\bibitem{Sat19a}
K.~Sato.
\newblock \textup{Ascending chain condition for $ F $-pure thresholds on a
  fixed strongly $ F $-regular germ}.
\newblock {\em Compositio Mathematica}, 155(6):1194--1223, 2019.

\bibitem{Sat19b}
K.~Sato.
\newblock \textup{Ascending chain condition for $F$-pure thresholds with fixed
  embedding dimension}.
\newblock {\em International Mathematics Research Notices}, 2019.

\bibitem{ST}
K.~Sato and S.~Takagi.
\newblock \textup{General hyperplane sections of threefolds in positive
  characteristic}.
\newblock {\em Journal of the Institute of Mathematics of Jussieu}, pages
  1--15, 2018.

\bibitem{Scho10}
H.~Schoutens.
\newblock {\em The use of ultraproducts in commutative algebra}, volume 1999.
\newblock Springer, 2010.

\bibitem{Sch09}
K.~Schwede.
\newblock \textup{$F$-adjunction}.
\newblock {\em Algebra $\&$ Number Theory}, 3(8):907--950, 2009.

\bibitem{Sch10}
K.~Schwede.
\newblock \textup{Centers of $F$-purity}.
\newblock {\em Mathematische Zeitschrift}, 265(3):687--714, 2010.

\bibitem{Sch11}
K.~Schwede.
\newblock \textup{Test ideals in non-$\mathbb{Q}$-Gorenstein rings}.
\newblock {\em Transactions of the American Mathematical Society},
  363(11):5925--5941, 2011.

\bibitem{ST14b}
K.~Schwede and K.~Tucker.
\newblock \textup{Test ideals of non-principal ideals: computations, jumping
  numbers, alterations and division theorems}.
\newblock {\em Journal de Math{\'e}matiques Pures et Appliqu{\'e}es},
  102(5):891--929, 2014.

\bibitem{Tak04}
S.~Takagi.
\newblock \textup{An interpretation of multiplier ideals via tight closure}.
\newblock {\em Journal of Algebraic Geometry}, 13(2):393--415, 2004.

\bibitem{Tak06}
S.~Takagi.
\newblock \textup{Formulas for multiplier ideals on singular varieties}.
\newblock {\em American Journal of Mathematics}, 128(6):1345--1362, 2006.

\bibitem{Tak13}
S.~Takagi.
\newblock \textup{Adjoint ideals and a correspondence between log canonicity
  and F-purity}.
\newblock {\em Algebra $\&$ Number Theory}, 7(4):917--942, 2013.

\bibitem{TW04}
S.~Takagi and K.-i. Watanabe.
\newblock \textup{On $F$-pure thresholds}.
\newblock {\em Journal of Algebra}, 282(1):278--297, 2004.

\end{thebibliography}
\end{document}